\documentclass[twoside,a4paper,reqno,12pt]{amsart}
\usepackage{amsfonts, amsbsy, amsmath, amssymb, latexsym}
\usepackage{mathrsfs,array}
\usepackage[top=1 in,right=1 in,bottom=1 in,left=1 in]{geometry}
\usepackage[utf8]{inputenc}
\usepackage{hyperref}
\usepackage{tikz-cd}
\usepackage[english]{babel}
\usepackage{graphicx}
\usepackage{mathtools}
\usepackage{floatflt}
\usepackage{multicol}
\usepackage{lipsum}
\usepackage{diagbox}
\usepackage{wrapfig}
\usepackage{relsize}
\usepackage{booktabs}
\usepackage{pgfplots}
\usepackage{pgfplotstable}
\usepackage{enumerate}
\usepackage{blindtext}
\usepackage{amssymb}
\usepackage{amsthm}
\usepackage[mathscr]{euscript}
\usepackage{enumitem}
\usepackage{tikz}
\pgfplotsset{compat=1.18} 

\usetikzlibrary{graphs,graphs.standard}

\usepackage{array,tabularx}

\usepackage[english]{babel}

\newtheorem{theorem}{Theorem}[section]
\newtheorem{lemma}[theorem]{Lemma}
\newtheorem{proposition}[theorem]{Proposition}
\newtheorem{corollary}[theorem]{Corollary}

\theoremstyle{definition}
\newtheorem{definition}[theorem]{Definition}
\newtheorem{example}[theorem]{Example}

\theoremstyle{remark}
\newtheorem{remark}[theorem]{Remark}

\numberwithin{equation}{section}

\newcommand{\K}{\mathbb{K}}
\newcommand{\Z}{\mathbb{Z}}

\newcommand{\N}{\mathbb{N}}

\newcommand{\inv}{^{-1}}
\newcommand{\negsigma}{\overline{\Sigma}}

\newcommand{\calH}{\mathcal{H}}

\newcommand{\pair}{(\kappa, \tau)}

\newcommand{\cc}{\pi_0}

\newcommand{\setcomp}{(V_1,\dots,V_p)}

\newcommand{\signedtriangle}[1][3.2]{
  \begin{tikzpicture}[baseline=0pt]
    \coordinate (A) at (0ex,0ex);
    \coordinate (B) at (#1ex,0ex);
    \coordinate (C) at (0.5*#1ex,{0.866*#1ex}); 

     \fill (A) circle (0.08*#1ex);
     \fill (B) circle (0.08*#1ex);
     \fill (C) circle (0.08*#1ex);
    \draw[thick, dotted] (A) -- (B);
    \draw[thick, dotted] (A) -- (C);
    \draw[thick] (B) -- (C);
  \end{tikzpicture}%
}

\newcommand{\twosignededges}[1][3.2]{%
  \begin{tikzpicture}[baseline=0pt]
    \coordinate (A) at (0ex,0ex);
    \coordinate (B) at (#1ex,0ex);
    \coordinate (C) at ({0.5*#1ex},{0.866*#1ex});

    \fill (A) circle (0.08*#1ex);
    \fill (B) circle (0.08*#1ex);
    \fill (C) circle (0.08*#1ex);

    \draw[thick] (A) -- (B);
    \draw[thick,dotted] (A) -- (C);
  \end{tikzpicture}
}

\newcommand{\twopositiveedges}[1][3.2]{
  \begin{tikzpicture}[baseline=0pt]
    \coordinate (A) at (0ex,0ex);
    \coordinate (B) at (#1ex,0ex);
    \coordinate (C) at (0.5*#1ex,{0.866*#1ex}); 

     \fill (A) circle (0.08*#1ex);
     \fill (B) circle (0.08*#1ex);
     \fill (C) circle (0.08*#1ex);
    \draw[thick,] (A) -- (B);
    \draw[thick,] (B) -- (C);
  \end{tikzpicture}%
}

\newcommand{\othersignedtriangle}[1][3.2]{\begin{tikzpicture}[baseline=0pt]

    \coordinate (A) at (0ex,0ex);
    \coordinate (B) at (#1ex,0ex);
    \coordinate (C) at (0.5*#1ex,{0.866*#1ex});

     \fill (A) circle (0.08*#1ex);
     \fill (B) circle (0.08*#1ex);
     \fill (C) circle (0.08*#1ex);

	\draw[thick] (A) to[bend left=20] (B);
	\draw[thick,dotted] (A) to[bend right=20] (B);

	\draw[thick,dotted] (A) -- (C);
	\draw[thick] (B) -- (C);

\end{tikzpicture}}

\newcommand{\positiveandnegativeedge}[1][3.2]{\begin{tikzpicture}[baseline=0pt]

    \coordinate (A) at (0ex,0.5ex);
    \coordinate (B) at (#1ex,0.5ex);

     \fill (A) circle (0.08*#1ex);
     \fill (B) circle (0.08*#1ex);

	\draw[thick] (A) to[bend left=30] (B);
	\draw[thick,dotted] (A) to[bend right=30] (B);

\end{tikzpicture}}

\newcommand{\positiveandnegativeedgeloop}[1][3.2]{\begin{tikzpicture}[baseline=0pt]

    \coordinate (A) at (0ex,0.5ex);
    \coordinate (B) at (#1ex,0.5ex);

     \fill (A) circle (0.08*#1ex);
     \fill (B) circle (0.08*#1ex);

	\draw[thick] (A) to[bend left=30] (B);
	\draw[thick,dotted] (A) to[bend right=30] (B);
    \draw[thick] (B) to[loop] (B);

\end{tikzpicture}}

\newcommand{\negativeedge}[1][3.2]{
  \begin{tikzpicture}[baseline=0pt]
    \coordinate (A) at (0ex,0.5ex);
    \coordinate (B) at (#1ex,0.5ex);
     \fill (A) circle (0.08*#1ex);
     \fill (B) circle (0.08*#1ex);

    \draw[thick,dotted] (A) -- (B);
  \end{tikzpicture}%
}

\newcommand{\positiveedge}[1][3.2]{
  \begin{tikzpicture}[baseline=0pt]
    \coordinate (A) at (0ex,0.5ex);
    \coordinate (B) at (#1ex,0.5ex);
     \fill (A) circle (0.08*#1ex);
     \fill (B) circle (0.08*#1ex);

    \draw[thick] (A) -- (B);
  \end{tikzpicture}%
}

\newcommand{\threepath}[1][1]{%
  \begin{tikzpicture}[baseline=0pt]
    \coordinate (A) at (0,0.1*#1);
    \coordinate (B) at (0.5*#1,0.1*#1);
    \coordinate (C) at (#1,0.1*#1);

    \fill (A) circle (0.08*#1);
    \fill (B) circle (0.08*#1);
    \fill (C) circle (0.08*#1);

    \draw[thick] (A) -- (B);
    \draw[thick] (B) -- (C);
  \end{tikzpicture}%
}

\newcommand{\exampletutte}[1][1]{%
  \begin{tikzpicture}[baseline=0pt]
    \coordinate (A) at (0,0.1*#1);
    \coordinate (B) at (0.5*#1,0.1*#1);
    \coordinate (C) at (#1,0.1*#1);

    \fill (A) circle (0.08*#1);
    \fill (B) circle (0.08*#1);
    \fill (C) circle (0.08*#1);

	\draw[thick] (A) to[bend left=30] (B);
	\draw[thick] (A) to[bend right=30] (B);
    \draw[thick] (B) -- (C);
    \draw[thick] (C) to[loop] (C);
    
  \end{tikzpicture}%
}

\newcommand{\exampletutteuno}[1][0.8]{%
  \begin{tikzpicture}[baseline=0pt]
    \coordinate (A) at (0,0.1*#1);
    \coordinate (B) at (0.5*#1,0.1*#1);
    \coordinate (C) at (#1,0.1*#1);

    \fill (A) circle (0.08*#1);
    \fill (B) circle (0.08*#1);
    \fill (C) circle (0.08*#1);

    \draw[thick] (B) -- (C);
    \draw[thick] (C) to[loop] (C);
    
  \end{tikzpicture}%
}

\newcommand{\exampletuttedos}[1][0.8]{%
  \begin{tikzpicture}[baseline=0pt]
    \coordinate (A) at (0,0.1*#1);
    \coordinate (B) at (0.5*#1,0.1*#1);
    \coordinate (C) at (#1,0.1*#1);

    \fill (A) circle (0.08*#1);
    \fill (B) circle (0.08*#1);
    \fill (C) circle (0.08*#1);

    \draw[thick] (C) to[loop] (C);
    
  \end{tikzpicture}%
}

\newcommand{\exampletuttetres}[1][0.8]{%
  \begin{tikzpicture}[baseline=0pt]
    \coordinate (A) at (0,0.1*#1);
    \coordinate (B) at (0.5*#1,0.1*#1);
    \coordinate (C) at (#1,0.1*#1);

    \fill (A) circle (0.08*#1);
    \fill (B) circle (0.08*#1);
    \fill (C) circle (0.08*#1);

	\draw[thick] (A) to[bend left=30] (B);
	\draw[thick] (A) to[bend right=30] (B);
    \draw[thick] (C) to[loop] (C);
    
  \end{tikzpicture}%
}

\newcommand{\vertex}[1][3.2]{
  \begin{tikzpicture}[baseline=0pt]
    \coordinate (A) at (0.5ex,0.5ex);
    \fill (A) circle (0.08*#1ex);

  \end{tikzpicture}%
}

\title{The combinatorial Hopf algebra of signed graphs}

\author{Jean-Christophe Aval}
\address{LaBRI, CNRS, Université de Bordeaux \\
Domaine Universitaire, 351, Cours de la Libération, 33405 Talence \\
France}
\email{aval@labri.fr}

\author{Raquel Melgar}
\address{LaBRI, Université de Bordeaux \\
Domaine Universitaire, 351, Cours de la Libération, 33405 Talence \\
France}
\email{raquel.melgar@labri.fr}

\thanks{}

\begin{document}

\begin{abstract}
The theory of combinatorial Hopf algebras (CHA) is a powerful framework for studying algebraic invariants of combinatorial objects.
For instance, Stanley's chromatic symmetric function and the chromatic polynomial arise naturally from the CHA of graphs.
In this paper, we define a CHA of signed graphs in order to study, in the same spirit, invariants of signed graphs such as the signed chromatic polynomial.
The key ingredient of our construction is a well-chosen character, from which the signed chromatic invariants are obtained.
Via the antipode, we obtain new proofs of combinatorial reciprocity results for these invariants.
Moreover, by deforming the character with an additional parameter, we incorporate bivariate invariants into this framework, both for classical and signed graphs.
In the classical case, this technique recovers the Tutte polynomial.
\end{abstract}

\maketitle
\section*{Introduction}
In the past years, the notion of combinatorial Hopf algebra (CHA) has proven its importance to define and study ``canonical'' algebraic invariants of combinatorial objects. 
This notion has been founded by Aguiar, Bergeron, and Sottile in \cite{ABS06}.
A CHA is a graded connected Hopf algebra endowed with a character (a multiplicative linear form).
The main result in \cite{ABS06} claims that the CHA of quasisymmetric functions ($QSym$) is the terminal object in the category of CHAs.
In other words, for any CHA there exists a unique morphism of CHA from it to $QSym$.
Typically, CHAs have bases indexed by combinatorial objects and thus, this canonical morphism is a natural way to associate a quasisymmetric function to a combinatorial object.

An interesting example of this phenomenon occurs in the case of graphs. In \cite{S94}, Schmitt defined the incidence Hopf algebra of graphs. When this Hopf algebra is equipped with the character defined by
\begin{equation}
    \zeta(G): = \begin{cases}
        1 \text{ if } E = \emptyset\\
        0\text{ otherwise}
    \end{cases}
\end{equation}
for any graph $G = (V,E)$, one obtains the CHA of graphs. The only morphism from the CHA of graphs to the CHA of quasisymmetric functions is the one that sends every graph to its chromatic symmetric function. This function was first defined by Stanley in \cite{S95} and has been the subject of an intense research activity in the last three decades (see for example \cite{aliste2014proper, shareshian2016chromatic, hikita2024proof}).

The CHA of graphs has been deeply studied in the last years.
For instance, it has been generalized to simplicial complexes by Benedetti, Hallam, and Machacek in \cite{BHM16}. And in \cite{AA23}, Aguiar and Ardila place the Hopf structure of graphs within the much more general framework of the Hopf monoid of generalized permutahedra.

Moreover, this CHA approach is convenient for obtaining combinatorial reciprocity results for these invariants via the antipode of the Hopf algebra structure, thus giving a common framework to formulas analogous to Stanley's chromatic reciprocity theorem.

Motivated by this, we define in this article the CHA of signed graphs in order to examine the algebraic invariants of signed graph coloring from the CHA perspective.

Signed graphs were first defined by Harary in the context of social psychology in \cite{H53}. A signed graph is a graph having a sign (+ or -) attached to each of its edges. Thus, if a classical graph can be used to represent a social network, a signed graph can represent liking (positive edges), disliking (negative edges), or indifference (no edges) among the people in the network.

The theory of signed graphs was developed in the 80s by Zaslavsky in a series of articles \cite{Z82, Z82+, Z91}.
He defined the matroid associated to a signed graph (also called the frame matroid). He also connected signed graphs to the theory of hyperplane arrangements; signed graphs represent subarrangements of the type B/D hyperplane arrangements. This viewpoint extends the classical correspondence between (unsigned) graphs and subarrangements of the braid arrangement (type A). Moreover, he introduced the notion of proper coloring of a signed graph, using colors in the relative integers. This  leads to the definition of the chromatic polynomials of a signed graph. Two different polynomials are defined depending on whether the color zero is allowed or not. 
These polynomials satisfy properties analogous to those of Birkhoff's chromatic polynomial.
For example, they satisfy the deletion-contraction recurrence relations, and in the case of the polynomial that avoids the zero color, it coincides with the characteristic polynomial of the associated matroid and hyperplane arrangement.
In \cite{BZ05}, Beck and Zaslavsky embed these polynomials in the context of Ehrhart theory as they are integer point counting polynomials.

These chromatic polynomials for signed graphs have been generalized to different algebraic invariants.
For example, various formal power series associated to a signed graph that specialize in the chomatic polynomial are defined in \cite{W97, KT21, coppola2023extension, aval2025chromatic}, and \cite{aval2025chromatic}. And we also mention a series of articles that seek to define polynomials in two or more variables that generalize the chromatic polynomials of a signed graph, see for instance \cite{kauffman1989tutte, Z95, beck2015bivariate, GLRV21}, and \cite{gao2022acyclic}.

A vast body of literature has been written on signed graphs (see \cite{recap} and the references therein).
Nevertheless, the theory of signed graph coloring has never been approached from the point of view of CHA. In this paper, we take a first step in this direction.

In Section \ref{sec: 1}, we review some necessary preliminaries on signed graphs and CHAs.
In Section \ref{sec: 2}, we define the combinatorial Hopf algebra of signed graphs.
The bialgebra structure is inherited from that of the incidence Hopf algebra of graphs and, together with the suitable definition of the character (Equation \ref{eq: def-signed-character}), one gets the CHA of signed graphs.
In Theorem \ref{th: canonical-series} we exhibit the canonical quasisymmetric function associated to a signed graph, that naturally specializes in the chromatic polynomial (Corollary \ref{coro:chromatic-signed}).
In Section \ref{sec: 3}, we use this point of view to obtain combinatorial reciprocity results for these two invariants.
Our proof (Corollary \ref{coro:recip}), uses a sign reversing involution argument and is different from the original one \cite[Theorem 3.5]{Z82}.
Finally, in Section \ref{sec: 4}, through the notion of deformation of the character, we are able to embed bivariate invariants in the framework of CHAs. 
These methods apply to both signed and unsigned graphs, and in the latter case, they allow us to obtain, for example, the Tutte polynomial (Subsection \ref{subsec: dichromatic}).
This notion of deformation of a character is inspired from the work of Benedetti, Hallam, and Machacek \cite{BHM16}.

\begin{remark}
As mentioned, in Zaslavsky's work (and in the literature in general) the colorings that allow the color zero and the colorings that do not are treated separately and lead to the definition of two different chromatic polynomials. In this work, we deal with colorings that do not use the color zero, that is, colorings in the color set $\Z\backslash\{0\}$. As we will see, these colorings fit really well in the context of CHA. Nevertheless, the colorings in the color set $\Z$ require a more special treatment. This circumstance is in fact interesting itself and it will be the subject of further study.
\end{remark}

\section{Preliminaries}
\label{sec: 1}
\subsection{Signed graphs}
We begin by introducing some preliminaries and notations for graphs and signed graphs. A \textit{graph} $G$ is a pair $(V,E)$, where $V$ is the set of vertices and $E$ is the set of edges. A graph may contain multiple edges (i.e., several edges between the same pair of vertices) and loops (edges that connect a vertex to itself). A \textit{signed graph} is a triple $\Sigma = (V,E,\sigma)$ where $(V,E)$ is a graph and $\sigma$ is the \textit{sign function}, a map $\sigma: E \rightarrow \{+,-\}$ that assigns to each edge either a $+$ or a $-$ sign. Signed graphs were introduced by Harary in \cite{H53}, where he also defined the notion of \textit{balance}. A signed graph $\Sigma$ is said to be \textit{balanced} if every cycle of $\Sigma$ contains an even number of negative edges. If a graph is not balanced we say it is \textit{unbalanced}. The following classical theorem provides a useful characterization for balanced signed graphs.

\begin{theorem}[\cite{H53}, Theorem 3]
\label{th: harary-decomposition-theorem}
A signed graph $\Sigma = (V,E,\sigma)$ is balanced if and only if there exists a partition of its vertex set $V = A \sqcup B$ such that all edges between vertices on the same subset are positive and all edges between vertices on different subsets are negative.    
\end{theorem}

This statement is illustrated in Figure~\ref{fig: decomp-simple}.
Throughout the article, a positive edge appears as a solid line,
and a negative edge as a dashed line.

\begin{figure}
    \centering
    \includegraphics[width=0.5\linewidth]{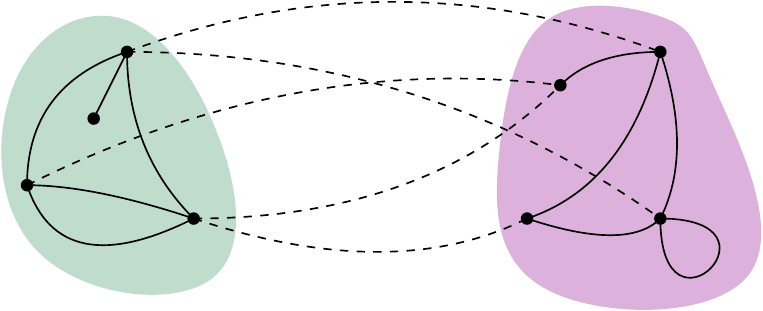}
    \caption{A balanced signed graph $\Sigma = (V,E,\sigma)$ and a partition $V = A \sqcup B$ as in Theorem \ref{th: harary-decomposition-theorem}.}
    \label{fig: decomp-simple}
\end{figure}

Given a graph $G = (V,E)$, a \textit{proper coloring} of $G$ is a map $\kappa: V \rightarrow \N = \{1,2,\dots\}$ verifying that for any edge $\{u,v\}\in E$, $\kappa(u) \neq \kappa(v)$. Graphs with loops do not admit any proper coloring. The \textit{chromatic polynomial} of $G$ is the polynomial $\chi_G$ verifying $$\chi_G(t) = \#\{\text{proper colorings }\kappa: V \rightarrow [t]:= \{1,\dots,t\}\}\text{ for any } t\in \N.$$

The definition of proper coloring for signed graphs first appeared in \cite{Z82}. Given a signed graph $\Sigma = (V,E,\sigma)$, a proper coloring of $\Sigma$ is a map $\kappa: V\rightarrow \{\pm1,\pm2,\dots\}$ such that for any edge $e = \{u,v\}\in E$, $\kappa(u) \neq \sigma(e) \kappa(v)$. In other words, adjacent vertices linked by a positive edge must be assigned different colors, while adjacent vertices linked by a negative edge cannot be assigned opposite colors. Note that if a signed graph has positive loops, it does not admit any proper coloring, and that negative loops do not impose any restrictions on the coloring. Given a signed graph $\Sigma = (V,E,\sigma)$, its \textit{balanced chromatic polynomial} is the polynomial $\chi^b_\Sigma$  satisfying 
$$\chi^b_\Sigma(2t) = \#\{\text{proper colorings } \kappa: V \rightarrow \pm[t]:=\{\pm1,\dots,\pm t\}\}\text{ for any } t\in \N.$$

Given a  signed graph $\Sigma = (V,E,\sigma)$ the \textit{incidence set} of $\Sigma$, denoted $In(\Sigma)$, is the set of pairs $(v,e)$ where $v\in V$ is an endpoint of $e \in E$. An \textit{orientation} of $\Sigma$ is a map $\tau: In(\Sigma) \rightarrow \{+,-\}$, such that
\begin{equation}
\label{eq: def-tau}
\sigma(e) = -\tau(u,e)\tau(v,e)
\end{equation}
for any edge $e = \{u,v\}\in E$. The interpretation of $\tau$ is as follows: if $\tau(v,e) = +$, then the incidence $(v,e)$ is directed towards the vertex $v$; if $\tau(v,e) = -$, then the incidence $(v,e)$ is directed away from $v$. Condition \eqref{eq: def-tau} implies that positive edges are oriented exactly as in unsigned graphs: each positive edge is directed from one endpoint to the other. In contrast, a negative edge is either oriented inward at both endpoints or outward at both endpoints. See Figure \ref{fig: example-orientation} for an example of orientation of a signed graph.

\begin{figure}[h]
    \centering
    \includegraphics[width=0.6\linewidth]{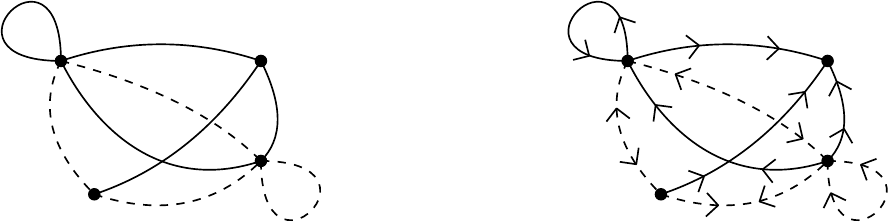}
    \caption{A signed graph $\Sigma$ (left) and an example of orientation $\tau$ of $\Sigma$ (right).}
    \label{fig: example-orientation}
\end{figure}

We say that an orientation $\tau$ of a signed graph $\Sigma$ is \textit{acyclic} if every closed walk in $\Sigma$ contains either a source or a sink. Equivalently, for every closed walk, there exists a vertex such that all edges of the path incident with that vertex are oriented either towards it or away from it. We say a coloring $\kappa: V\rightarrow\N$ and an orientation $\tau$ of $\Sigma$ are \textit{compatible on an edge} $e=\{u,v\}\in E$ if \begin{equation}
    \label{eq: def-compatib}
    \tau(u,e)\kappa(u) + \tau(v,e)\kappa(v) \leq 0.
\end{equation} And we say $\kappa$ and $\tau$ are \textit{compatible} if they are compatible on every $e\in E$. Note that if a coloring $\kappa$ is proper on an edge $e = \{u,v\}$, i.e. $\kappa(u)\neq\sigma(e) \kappa(v)$, then there exist unique values of $\tau(u,e)$ and $\tau(u,v)$ such that $\tau$ is an orientation compatible with $\kappa$. If $\kappa$ is a proper coloring of $\Sigma$, we call the resulting orientation $\tau$ the orientation \textit{induced} by $\kappa$.
The following proposition appears in \cite{Z82}.
\begin{proposition}
    \label{prop: proper-induces-acyclic}
    Let $\Sigma = (V,E,\sigma)$ be a signed graph and $\kappa$ a proper coloring of it, then $\kappa$ induces an acyclic orientation on $\Sigma$. Reciprocally, given an acyclic orientation $\tau$ there always exists a proper coloring $\kappa$ inducing $\tau$.
\end{proposition}
\subsection{Combinatorial Hopf Algebras }
In this subsection we review some facts and definitions about combinatorial Hopf algebras (CHA). We encourage the reader to visit \cite{ABS06} and \cite{GR14} for a more detailed study of this topic.

Throughout this article, let $\K$ be a field of characteristic zero. A \textit{bialgebra} $\mathcal{H}$ is a $\K$-vector space equipped with four $\K$-linear maps $\mu:\calH\otimes \calH \rightarrow \calH$, $\Delta:\calH\rightarrow\calH\otimes\calH$, $u:\K\rightarrow\mathcal{H}$, and $\epsilon:\calH\rightarrow\K$ called  respectively \textit{product}, \textit{coproduct}, \textit{unit}, and \textit{counit}, verifying:
\begin{enumerate}
    \item $(\calH,\mu,u)$ is an associative algebra;
    \item $(\calH, \Delta, \epsilon)$ is a coassociative coalgebra;
    \item $\Delta$ and $\epsilon$ are morphisms for the algebra structure $(\calH,\mu,u)$.
\end{enumerate}

A bialgebra is a \textit{Hopf algebra} if there exists a $\K$-linear map $S:\calH \rightarrow \calH$ called \textit{antipode} satisfying $$\mu \circ(S\otimes id)\circ \Delta = u \circ \epsilon = \mu \circ(id \otimes S)\circ \Delta.$$

A bialgebra is \textit{graded} if $\calH = \bigoplus_{d\geq 0} \calH_d$ and the grading is compatible with the four maps. That is $\mu(\calH_n \otimes \calH_m)\subseteq \calH_{n+m}$, $\Delta(\calH_d)\subseteq\sum_{n+m = d}\calH_n \otimes \calH_m$, $u(\K)\subseteq \calH_0$, and $\epsilon(\calH_d) = 0$ for $d\geq 1$. The elements in each $\calH_d$ are called \textit{homogeneous components of degree d}. We say a graded bialgebra is \textit{connected} if $\dim_{\K}(\calH_0) = 1$. Graded connected bialgebras are always Hopf algebras and the antipode can be computed recursively.

A \textit{character} of a graded connected Hopf algebra is a map $\zeta:\calH \rightarrow \K$ that is a morphism of algebras. The \textit{convolution product} $\ast$ of two characters $\zeta_1$ and $\zeta_2$ is the character defined by $$\zeta_1 \ast \zeta_2 = m \circ (\zeta_1 \otimes \zeta_2)\circ \Delta$$ where $m$ denotes the product in $\K$. Given a graded connected Hopf algebra, the convolution makes the set of characters into a group with unit $\epsilon$ and inverse given by $\zeta\inv = \zeta \circ S$.

Finally, a \textit{combinatorial Hopf algebra} is a pair $(\calH, \zeta)$ where $\calH$ is a graded connected Hopf algebra such that $dim_{\K}(\calH_d)<\infty$ for all $d\geq 0$ and $\zeta$ is a character. This notion was first defined in \cite{ABS06}.

A CHA of great importance is that of \textit{quasisymmetric functions}. Given any composition $\alpha = (\alpha_1,\dots,\alpha_{\ell(\alpha)})$, the \textit{monomial quasisymmetric function} $M_{\alpha}$ is the element of $\K[[x_1,x_2,\dots]]$ defined as: 
\begin{equation}
    M_\alpha : = \sum_{i_1<\cdots<i_{\ell(\alpha)}}x_{i_1}^{\alpha_1}\cdots x_{i_{\ell(\alpha)}}^{\alpha_{\ell(\alpha)}}.
\end{equation}

 The algebra of quasisymmetric functions is the graded $\K$-vector space $$QSym = \bigoplus_{d\geq 0}QSym_d,$$ where, for each $d\geq 0$, $QSym_d : = span_\K\{M_\alpha:\alpha \models d\}$. Endowed with its standard Hopf algebra structure (see \cite{GR14}), $QSym$ becomes a graded connected Hopf algebra. The CHA of quasisymmetric functions consists in the pair $(QSym, \zeta_{\mathcal{Q}})$ where $\zeta_{\mathcal{Q}}(f) = f(1,0,0,\dots)$ for any $f(x_1,x_2,\dots)\in QSym$.

 Given a composition $\alpha$, denote by $\lambda(\alpha)$ the partition obtained by rearranging the parts of $\alpha$. Then the \textit{monomial symmetric function} $m_\lambda$ is defined as:
 \begin{equation}
     m_\lambda: = \sum_{\alpha:~\lambda(\alpha) = \lambda} M_\alpha.
 \end{equation} The $\K$-vector space spanned by the functions $m_\lambda$, as $\lambda$ runs over all partitions, is a Hopf subalgebra of $QSym$. This Hopf algebra, denoted by $Sym$ is the Hopf algebra of symmetric functions. And $(Sym, \zeta_{\mathcal{Q}}|_{Sym})$ is the CHA of symmetric functions.

The main result in \cite{ABS06} is that $(QSym, \zeta_{\mathcal{Q}})$ is the terminal object in the category of combinatorial Hopf algebras:

\begin{theorem}[\cite{ABS06}, Theorem 4.1]
Given a CHA $(\mathcal{H},\zeta)$ there exists a unique morphism of CHA
\begin{equation}
\label{eq: canonical-morphism}
\Psi: (\mathcal{H},\zeta) \rightarrow (QSym,\zeta_{\mathcal{Q}}).\end{equation}
\end{theorem}
Moreover, it is proved in the same article that if $(\mathcal{H},\zeta)$ is cocommutative, then the image of $\Psi$ is in  $(Sym, \zeta_{\mathcal{Q}}|_{Sym})$ and that the CHA of symmetric functions is the terminal object in the category of cocommutative CHAs.

Due to this theorem, there is a canonical way to assign a quasisymmetric function to any element $h\in \mathcal{H}$. And, since it is possible to define combinatorial Hopf algebras whose elements are combinatorial objects, we can use the morphism $\Psi$ as a tool for constructing algebraic invariants of combinatorial objects.

\subsection{The CHA of graphs}
We now show how the above framework applies to the CHA of graphs. We consider the $\K$-vector space $\mathcal{G} = \bigoplus_{d\geq 0}\mathcal{G}_d$ where $\mathcal{G}_d$ is the linear span of isomorphism classes of graphs on $d$ vertices. Given a graph $G=(V,E)$ and a subset $S\subseteq V$, let $G|_S$ denote the subgraph of $G$ induced by $S$. Define the following maps on the basis elements: 
\begin{itemize}
    \item product: $\mu(G\otimes H): = G\sqcup H$, the disjoint union of the graphs $G$ and $H$;
    \item coproduct: $\Delta(G) := \sum_{S\subseteq V}G|_S\otimes G|_{V\backslash S}$ where $V$ is the vertex set of $G$;
    \item unit: $u(1):= \emptyset$;
    \item counit: $\epsilon(G) := 1$ if $G =\emptyset$ and $\epsilon(G) := 0$ otherwise, where $\emptyset$ denotes the empty graph.

\end{itemize}
These four maps make $\mathcal{G}$ a graded connected Hopf algebra that was first considered by Schmitt \cite{S94}. For any graph $G = (V,E)$, define the character 
\begin{equation}
\label{eq: zeta-graphs}
    \zeta(G):=\begin{cases}
        1 ~\text{ if } E=\emptyset\\
        0~\text{ otherwise}
    \end{cases}
\end{equation}
and extend it linearly to all elements of $\mathcal{G}$. Then the pair $(\mathcal{G},\zeta)$ is the combinatorial Hopf algebra of graphs.

Given a graph $G = (V,E)$ with vertex set $V = \{v_1,\dots,v_d\}$ its \textit{chromatic symmetric function} $X_{G}$ is defined as 
\begin{equation}
\label{eq: chrom-symm-funct}
    X_G(x_1,x_2,\dots) := \sum_{\kappa \text{ proper}} x_{\kappa(v_1)}\cdots x_{\kappa(v_d)}
\end{equation}
where the sum runs over all proper colorings $\kappa:V\rightarrow \N$ of $G$.
This invariant was introduced by Stanley in \cite{S95}, since then, it has been the focus of extensive research.

Let $\Psi: (\mathcal{G},\zeta)\rightarrow (QSym,\zeta_{\mathcal{Q}})$ be the unique CHA morphism between the CHA of graphs and the CHA of quasisymmetric functions. We know that, as $(\mathcal{G},\zeta)$ is cocommutative, the image of $\Psi$ lives in $Sym$. For any graph $G$, $\Psi(G) = X_G$, the chromatic symmetric function of $G$.

In this way, Stanley's chromatic symmetric function appears as a canonical object arising from the study of the CHA of graphs. Moreover, the chromatic polynomial also arises naturally within this framework. Denote $\phi_t$ the map sending the first $t$ variables of any $f\in QSym$ to $1$ and the rest to $0$. That is
$$\phi_t(f) = f(x_1,x_2,\dots) |_{\substack{x_1=x_2=\cdots =x_t=1,\\x_{t+1}=x_{t+2}= \cdots =0}} =f(\underbrace{1,1,\ldots,1}_{t\text{ ones}},0,0,\ldots)$$ for any $f\in QSym$. From \eqref{eq: chrom-symm-funct} it is not difficult to check that for any graph $G = (V,E)$, $$\phi_t \circ X_G = \# \{\text{proper colorings }\kappa: V \rightarrow [t]\} = \chi_G(t),$$ the chromatic polynomial of $G$. The map $\phi_t$ is a CHA morphism between the CHA of quasisymmetric functions and the CHA of polynomials in one variable (known as the binomial CHA). Given any composition $\alpha = (\alpha_1,\dots,\alpha_{\ell(\alpha)})$, $\phi_t(M_\alpha) = \binom{t}{\ell(\alpha)}$, a polynomial of degree $\ell(\alpha)$ in $t$.

In general, given any CHA $(\mathcal{H},\zeta)$, we will be interested in the composition $\phi_t \circ \Psi$ as a way to associate a polynomial to each element $h\in \mathcal{H}$. In fact, for every $h\in \mathcal{H}$, this polynomial can be obtained as the $t^{th}$ convolution power of the character. That is $$\phi_t\circ\Psi(h) = \zeta ^{\ast t}(h) = m^{(t-1)}\circ \zeta^{\otimes t} \circ \Delta ^{(t-1)}(h).$$

\section{Definition of the CHA of signed graphs}
\label{sec: 2}
Given a signed graph $\Sigma = (V,E,\sigma)$ we will denote $\negsigma$ the signed graph obtained by reversing the sign of every edge in $\Sigma$. That is $\negsigma : = (V,E,-\sigma)$. For any connected signed graph $\Sigma$ define the map 
\begin{equation}
\label{eq: def-signed-character}
\zeta(\Sigma) := \begin{cases} 
				2\quad & \text{ if } \negsigma \text{ is balanced} \\
				0\quad & \text{ otherwise.}
\end{cases}
\end{equation}

Now extend multiplicatively the map $\zeta$ to all signed graphs. Thus, given a signed graph $\Sigma$ with $c$ connected components one obtains $\zeta (\Sigma) = 2^c$ if $\negsigma$ is balanced and 0 otherwise. 
\begin{lemma}
\label{lemma: decomposition}
Let $\Sigma = (V,E,\sigma)$ be a signed graph. Then $\zeta(\Sigma)$ counts the number of decompositions $V = A \sqcup B$ in which every edge whose endpoints both lie in $A$ or both lie in $B$ is negative, and every edge with one endpoint in $A$ and one in $B$ is positive.
\end{lemma}

\begin{proof}

It is clear that it suffices to prove it for connected graphs. By Harary's characterization of balanced signed graphs (Theorem~\ref{th: harary-decomposition-theorem}), we deduce that such a decomposition only exists if $\negsigma$ is balanced. Now, it suffices to show that there exist exactly two of these decompositions. If $\negsigma$ is balanced, we choose and fix $v\in V$. Let us first suppose $v\in A$. For any other $u\in V$, consider a path between $u$ and $v$. To ensure the desired properties of the decomposition, we have to set $u \in A$ if  the path has an even number of positive edges and $u \in B$ otherwise. Note that this assignment does not depend on the choice of the path; if $\negsigma$ is balanced, every closed path has an even number of positive edges. In this way, we obtain a valid decomposition; and exactly another one if we suppose $v\in B$. 
\end{proof}

We now define a structure of CHA for signed graphs. Consider the $\K$-vector space $\mathcal{S} = \bigoplus_{d\geq 0}\mathcal{S}_d$ where each $\mathcal{S
}_d$ is the linear span of ismorphism classes of signed graphs on $d$ vertices. The Hopf algebra structure is inherited from that of Schmitt for unsigned graphs. That is:
\begin{itemize}
    \item product: $\mu(\Sigma_1\otimes \Sigma_2): = \Sigma_1\sqcup \Sigma_2$, the disjoint union of the signed graphs $\Sigma_1$ and $\Sigma_2$;
    \item coproduct: $\Delta(\Sigma) := \sum_{S\subseteq V}\Sigma|_S\otimes \Sigma|_{V\backslash S}$ where $V$ is the vertex set of $\Sigma$;
    \item unit: $u(1):= \emptyset$;
    \item counit: $\epsilon(\Sigma) := 1$ if $\Sigma =\emptyset$ and $\epsilon(\Sigma) := 0$ otherwise, where $\emptyset$ denotes the empty signed graph.
\end{itemize}

\begin{example} Let $s_1 = \othersignedtriangle$ and $s_2 = 3 \cdot \signedtriangle -2\cdot \positiveandnegativeedgeloop$ be two elements of $\mathcal{S}$. Then the product of $s_1$ and $s_2$ is $$\mu(s_1 \otimes s_2)=3\cdot (\othersignedtriangle ~ \signedtriangle) - 2 \cdot (\othersignedtriangle ~ \positiveandnegativeedgeloop)$$ and the coproduct of $s_1$ is 
$$\begin{aligned}
    \Delta(s_1) ={}& \othersignedtriangle \otimes \emptyset 
    ~+~ \vertex \otimes \positiveandnegativeedge
    ~+~ \vertex \otimes \positiveedge
    ~+~ \vertex \otimes \negativeedge\\
    &+~ \positiveandnegativeedge \otimes \vertex
    ~+~ \positiveedge \otimes \vertex
    ~+~ \negativeedge \otimes \vertex
    ~+~ \emptyset \otimes \othersignedtriangle .
    \end{aligned}$$

\end{example}

Clearly, these four maps endow $\mathcal{S}$ with a structure of  graded connected Hopf algebra. The proof is a straightforward adaptation of the corresponding proof for (unsigned) graphs.

\begin{definition}
    The combinatorial Hopf algebra of signed graphs is the pair $(\mathcal{S},\zeta)$ where the character $\zeta$ is the function defined in \eqref{eq: def-signed-character}.
\end{definition}

\begin{definition}
    Given a signed graph $\Sigma = (V,E,\sigma)$, define the \textit{chromatic (unsigned) symmetric function of $\Sigma$} as
    \begin{equation}
        X_\Sigma^\parallel:=\sum_{\kappa \text{ proper}}x_{|\kappa(v_1)|} \cdots x_{|\kappa(v_d)|}
    \end{equation}
    where $V = \{v_1,\dots,v_d\}.$
\end{definition}

When considering a signed graph $\Sigma = (V,E,\sigma)$ such that $\sigma(e) = +$ for all $e\in E$, that is, an \textit{all positive} signed graph, we may see it as a classical graph $G = (V,E)$. Note that in this case the chromatic symmetric function of $G$ and the chromatic (unsigned) symmetric function of $\Sigma$ are not the same. For instance, if $\Sigma = \positiveedge$ is the graph with two vertices and one positive edge between them, $X_\Sigma^\parallel = 2 M_2 + 8 M_{11}$ and $X_{\positiveedge} = 2 M_{11}$.

In \cite{W97}, \cite{KT21}, \cite{coppola2023extension} the authors study the \textit{chromatic signed symmetric function of $\Sigma$}, defined as:
\begin{equation}
    X_{\Sigma}(x_1,x_{-1},x_2,x_{-2},...): = \sum_{\kappa \text{ proper}} x_{\kappa(v_1)}\cdots x_{\kappa(v_d)}.
\end{equation}
 This function is a power series in the set of variables $(x_1,x_{-1},x_2,x_{-2},\dots)$. It lives in a space of functions called \textit{the algebra of signed symmetric functions}. The chromatic (unsigned) symmetric function $X_\Sigma^\parallel$ of a signed graph $\Sigma$ corresponds to the image of $X_\Sigma$ under the map sending $x_{-i}$ to $x_i$ for all $i\in \N$.
 
Before stating the theorem we shall give some intuition on the character $\zeta$ for signed graphs. First, note that in the case of the CHA of graphs, the usual character evaluates
whether a graph $G$ admits a monochromatic proper coloring or not. In fact, it counts the number of proper monochromatic colorings of $G$. Note that the character $\zeta$ plays exactly the same role for signed graphs if we consider a ``monochromatic" signed coloring to be a signed coloring using the color set $\{\pm 1\}$. Whenever one has a signed graph $\Sigma = (V,E,\sigma)$ and a decomposition $V = A \sqcup B$ as in Lemma \ref{lemma: decomposition}, the coloring  $\kappa |_A = 1$, $\kappa|_B = -1$ is clearly proper. When one exchanges the roles of $1$ and $-1$, one finds the only other monochromatic proper signed coloring. This observation is key in the proof of the following theorem.

\begin{theorem}
\label{th: canonical-series}
Let $\Psi$ be the unique morphism of combinatorial Hopf algebras between the combinatorial Hopf algebra of signed graphs $(\mathcal{S},\zeta)$ and $(QSym, \zeta_{Q})$. Then for any signed graph $\Sigma$
\begin{equation}
\label{eq: expression-psi}
\Psi(\Sigma) = \sum_{\kappa \text{ proper}}x_{|\kappa(v_1)|} \cdots x_{|\kappa(v_d)|} = X_\Sigma^\parallel
\end{equation}
where $\{v_1,\dots, v_d \}$ is the vertex set of $\Sigma$. 
\end{theorem}
 \begin{proof}
 By applying the explicit formula for $\Psi(\Sigma)$ in \cite{ABS06}, one obtains 
 
  \begin{equation}
  \label{eq: expansion-de-psi}
 \Psi(\Sigma) = \sum_{\alpha \models d}\zeta_\alpha (\Sigma) M_\alpha
 \end{equation}
 where if $\alpha = (\alpha_1,\dots,\alpha_p)$,
 \begin{align}
 \label{eq: zeta-alpha}
 \zeta_\alpha (\Sigma) := \sum_{\substack{V = V_1 \sqcup \cdots \sqcup V_p \\|V_i| = \alpha_i~ \forall i}} \zeta(\Sigma\mid_{V_1})\cdots \zeta(\Sigma\mid_{V_p}).
 \end{align}
 
By Lemma \ref{lemma: decomposition}, $\zeta_\alpha(\Sigma)$ is the number of ordered decompositions of $V$ as $$V = V_1 \sqcup \cdots \sqcup V_p = (A_1 \sqcup B_1) \sqcup \cdots \sqcup (A_p \sqcup B_p)$$ where: 
\begin{itemize}
\item $|A_i \sqcup B_i| = \alpha_i$ for all $i = 1,\dots,p$;
\item if $e$ is an edge whose endpoints are both in some $A_i$ or some $B_i$ then $\sigma(e)=-$;
\item if $e$ is an edge with one endpoint in $A_i$ and one endpoint in $B_i$ for some $i$, then $\sigma(e) = +$. 
\end{itemize}

\begin{figure}[h]
    \centering
    \includegraphics[width=0.3\linewidth]{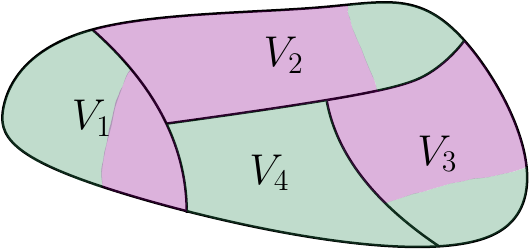}
    \caption{Decomposition $V = V_1 \sqcup \cdots \sqcup V_p$} where each block $V_i$ is decomposed as in Figure \ref{fig: decomp-simple}.
    \label{fig:placeholder}
\end{figure}

For every ordered decomposition as above and positive integers $j_1<\cdots <j_p$ there is a proper coloring $\kappa$ of $\Sigma$ given by $\kappa \vert _{A_{i}} = j_i$ and $\kappa \vert_{B_{i}} = - j_i$ for $i = 1,\dots,p$. This assignment is a bijection between decompositions of this type and proper colorings of $\Sigma$ verifying $|\kappa\inv (j_i)| + |\kappa\inv(-j_i)| = \alpha_i$ for $i = 1,\dots,p$. Thus we obtain the desired equality.
 \end{proof}

And the balanced chromatic polynomial is easily obtained by specialization. 
\begin{corollary} 
\label{coro:chromatic-signed}
Let $\Sigma$ be a signed graph, then 
\begin{equation}
\chi_\Sigma ^b (2t) = \phi_t\circ \Psi (\Sigma) = \zeta^{\ast t}(\Sigma).
\end{equation}
\end{corollary}

\begin{proof}
In view of \eqref{eq: expression-psi}, setting the first $t$ variables of $\Psi(\Sigma)$ equal to one and the remaining variables equal to zero yields the number of proper colorings of $\Sigma$ using colors in the set $\{\pm1,\dots,\pm t\}$.
\end{proof}

\begin{remark}
We may deduce from \eqref{eq: expansion-de-psi} that the function $\chi_\Sigma^b(z)$ is indeed a monic polynomial of degree $|V|$. To see this, note that $\Phi_t(M_\alpha) = \binom{t}{\ell(\alpha)}$ where $\ell(\alpha)$ is the length of the composition $\alpha$. This binomial coefficient is a polynomial in $t$ of degree $\ell(\alpha)$ and each $\zeta_\alpha(\Sigma)$ is divided by $2^{\ell(\alpha)}$. To obtain the leading monomial, consider the composition $\alpha = (1,\dots,1)$ and observe that $\zeta_\alpha(\Sigma) = 2^{\ell(\alpha)}=2^{|V|}$.
\end{remark}

\begin{example} Consider $\Sigma = \signedtriangle$. The terms $\zeta_\alpha(\Sigma)$ for $\alpha \models 3$ are:
\begin{itemize}
\item $\zeta_3(\Sigma) = \zeta(\Sigma) = 0$
\item $\zeta_{12}(\Sigma) = \zeta_{21}(\Sigma) = \zeta(\negativeedge)\zeta(\vertex) + \zeta(\negativeedge)\zeta(\vertex) + \zeta(\positiveedge)\zeta(\vertex) =  12$
\item $\zeta_{111}(\Sigma) = \zeta(\vertex)^3 = 8.$
\end{itemize}

Thus the chromatic (unsigned) symmetric function of $\Sigma$ is
\[ X_\Sigma^\parallel = \Psi(\Sigma) = 12 M_{12} + 12M_{21} + 8M_{111},\]
by composing with $\phi_t$, we obtain the following expression for the balanced chromatic polynomial $$\chi_\Sigma ^b (2t) = \phi_t \circ \Psi(\Sigma) = 24 \binom{t}{2} + 8\binom{t}{3}.$$
\end{example}

\begin{remark} 
Zaslavsky's work \cite{Z82+} includes the notion of \textit{switching}. 
To switch a signed graph with respect to a vertex means to change the signs of all the edges incident to that vertex. 
By performing a sequence of such operations at different vertices, one obtains the switching equivalence class of the signed graph.
It should be clear that all the tools and results in the present work are invariant under switching.
This is natural, since two signed graphs in the same switching class share the same balanced chromatic polynomial and the same chromatic unsigned symmetric function.
It appears in \cite[Corollary 3.3]{Z82+} that a signed graph is balanced if and only if it is switching equivalent to a graph with only positive edges.
Consequently, the character $\zeta$ may be interpreted for connected signed graphs as the function that sends a graph $\Sigma$ to $2$ if $\Sigma$ is in the switching class of a graph with only negative edges.
But in our presentation, as it is usually the case, we consider signed graphs rather than equivalence classes for the switching.
\end{remark}

\section{Reciprocity results}
\label{sec: 3}
The results in this section are not new in the literature, but they demonstrate the power of the CHA perspective for obtaining combinatorial reciprocity results via the antipode. The main idea is that, given a morphism of combinatorial Hopf algebras, the antipode must be preserved, which means that the following diagram is commutative. 

\begin{equation}
\label{diag: antipode}
\begin{tikzcd}
	{(\mathcal{S},\zeta)} && {(Sym,\zeta_{\mathcal{Q}}|_{Sym})} && {(\mathbb{K}[t],\zeta)} \\
	\\
	{(\mathcal{S},\zeta)} && {(Sym,\zeta_{\mathcal{Q}}|_{Sym})} && {(\mathbb{K}[t],\zeta)}
	\arrow["\Psi", from=1-1, to=1-3]
	\arrow["S"{description}, from=1-1, to=3-1]
	\arrow["{\phi_t}", from=1-3, to=1-5]
	\arrow["S"{description}, from=1-3, to=3-3]
	\arrow["S"{description}, from=1-5, to=3-5]
	\arrow["\Psi", from=3-1, to=3-3]
	\arrow["{\phi_t}", from=3-3, to=3-5]
\end{tikzcd}
\end{equation}

Let $\omega$ denote the usual involution of $Sym$, the algebra of symmetric functions. That is, $\omega$ is the algebra automorphism determined by $\omega(e_d)=h_d$, where $e_d$ and $h_d$ denote respectively the $d$-th elementary and complete homogeneous symmetric functions.
 If $f$ is an homogeneous symmetric function of degree $d$, the antipode $S$ in $Sym$ satisfies \begin{equation} \label{eq: antipode-omega}S(f) = (-1)^d \omega(f).\end{equation} We encourage the reader to consult \cite{GR14} for a detailed review of these questions. Our aim is to obtain a combinatorial interpretation of the image of $X_{\Sigma}^\parallel$ under this involution.

\begin{remark}
As the Hopf Algebra structure of $\mathcal{S}$ is the same as $\mathcal{G}$, the antipode is the same. We will use the formula for the antipode given directly by Takeuchi's formula
\begin{equation}
\label{eq: takeuchi}
S(\Sigma) = \sum_{p\geq 0}(-1)^p\sum_{V = V_1 \sqcup \cdots \sqcup V_p} \Sigma |_{V_1} \sqcup \cdots \sqcup \Sigma|_{V_p}\text{ for any signed graph }\Sigma
\end{equation}
 and not the cancellation free formula of Humpert and Martin \cite{HM12}. It might be interesting to find proofs using this formula.
\end{remark}

The following theorem is a generalization of the analog result for classical graphs due to Stanley \cite[Theorem 4.2]{S95}. And it is in fact the specialization in $x_i \mapsto x_{-i}$ of one of the main result in Wolfgang's thesis \cite{W97}. 
His approach is different, and relies essentially on the enumeration of lattice points in hyperplane arrangements.
We also mention here that for the polynomial formulation (Corollary~\ref{coro:recip}), the original proof due to Zaslavsky \cite{Z82} is also different, since it is based on a recursive argument.

\begin{theorem} Let $\Sigma$ be a signed graph then
\begin{equation}
\omega(X_{\Sigma}^\parallel(x_1,x_2,...)) = \sum_{\substack{\pair}}x_{|\kappa(v_1)|}\cdots x_{|\kappa(v_d)|}
\end{equation}
where the sum runs over the pairs $(\kappa, \tau)$ where $\tau$ is an acyclic orientation of $\Sigma$ and $\kappa$ is a coloring compatible with $\tau$. 
\end{theorem}
 
\begin{proof}
Since the diagram in \ref{diag: antipode} is commutative, we have $S(\Psi(\Sigma)) = \Psi(S(\Sigma))$. Then, applying \eqref{eq: antipode-omega} on the left-hand side and \eqref{eq: takeuchi} on the right-hand side  one obtains
\begin{equation}
\label{eq: first-equalities}
(-1)^d \omega(X_{\Sigma}^\parallel) = \sum_{p\geq 0}(-1)^p\sum_{\substack{V = V_1 \sqcup \cdots \sqcup V_p}} X^\parallel_{\Sigma |_{V_1}} \cdots X^\parallel_{\Sigma|_{V_p}} = \sum_{p\geq 0}(-1)^p\sum_{\substack{V = V_1 \sqcup \cdots \sqcup V_p \\ \kappa:~\kappa|{V_i} \text{ proper}}} x^{|\kappa|}.
\end{equation}

Now, consider any set composition of the vertex set $\setcomp$ and a coloring $\kappa$ of $\Sigma$ such that $\kappa |_{V_i}$ is proper for $i = 1,\dots,p$. We want to deduce from the pair $(\setcomp, \kappa)$ an orientation  $\tau$ of $\Sigma$. To do this we first derive a {\em proper} coloring $\kappa'$ of $\Sigma$. We consider 
$n=\max_{u\in V}|\kappa(u)|$, 
and for $u\in V_i$, we set
\begin{equation}
\label{eq: def-coloring}
    \kappa'(u) := \begin{cases} 
				(\kappa(u) -1 )\cdot n + i \quad & \text{ if } \kappa(u) >0 \\
				(\kappa(u) +1 )\cdot n - i \quad & \text{ if } \kappa(u) <0.
\end{cases}
\end{equation}

We observe that $\kappa'$ is a proper coloring of $\Sigma$:
\begin{itemize}
\item inside any $V_i$ because $\kappa$ is proper on $V_i$;
\item on any edge between two different subsets $V_i$ and $V_j$ because $\kappa '$ has different absolute values modulo $n$.
\end{itemize}

\begin{figure}
    \centering
    \includegraphics[width=0.7\linewidth]{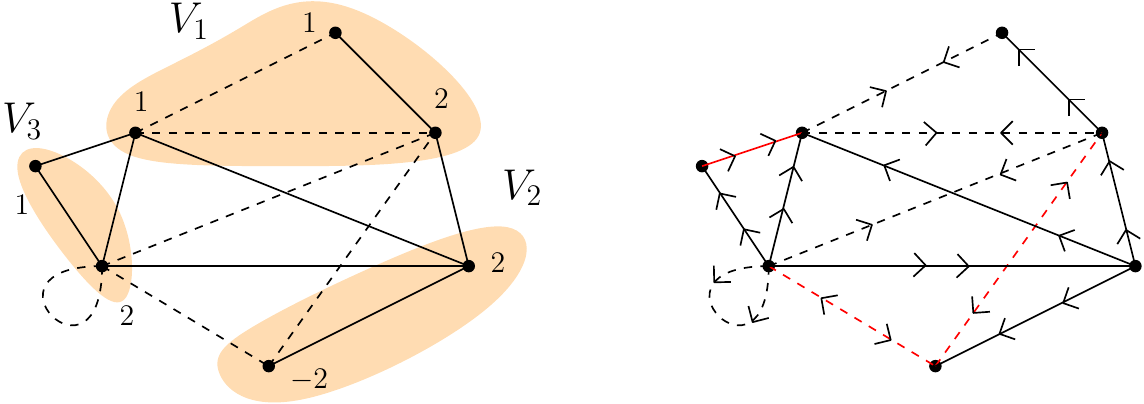}
    \caption{A pair $((V_1,\dots,V_p),\kappa)$ such that $\kappa|_{V_i}$ is a proper for every $i$ (left) and the acyclic orientation $\tau$ induced by this pair (right). On properly colored edges, the orientation is imposed by the coloring, whereas on the edges in the impropriety set (shown in red), the orientation is determined by the partition $(V_1,\dots,V_p)$.}
    \label{fig:placeholder}
\end{figure}
Thus, by Proposition \ref{prop: proper-induces-acyclic}, $\kappa '$ induces an acyclic orientation $\tau$ on $\Sigma$. Moreover, on any edge where $\kappa$ is proper, $\kappa'$ induces the same orientation as $\kappa$. In this way, we define the orientation induced by a pair $(\setcomp, \kappa)$. Note that this assignment is surjective, every acyclic orientation can be obtained in this way. To see this, consider any acyclic orientation $\tau$, and choose $\kappa$ any proper coloring compatible with $\tau$ (thanks to Proposition \ref{prop: proper-induces-acyclic}). Then for any composition on the vertex set $\setcomp$, the pair $(\setcomp, \kappa)$ induces $\tau$. This assignment is by far not injective, there are (many) cancellations in the formula, and we now compute them.

The value in \eqref{eq: first-equalities} is then equal to 
$$\sum_{(\kappa, \tau)} \sum_{p \geq 0} (-1) ^p \sum_{\substack{V = V_1\sqcup \cdots \sqcup V_p \\ \kappa: ~ \kappa|_{V_i} \text{ proper}\\ \text{ inducing  }\tau}}x^{|\kappa |}.$$

To conclude the proof we have to show that the cancellations give us: 
$$\sum_{p \geq 0} (-1) ^p \sum_{\substack{V = V_1\sqcup \cdots \sqcup V_p \\ \kappa: ~ \kappa|_{V_i} \text{ proper}\\ \text{ inducing  }\tau}} 1 = (-1) ^d.$$

To do this, we define a map $\theta: \Pi_{(\kappa,\tau)} \rightarrow \Pi_{(\kappa,\tau)}$  where $$\Pi_{(\kappa,\tau)} := \{\setcomp:~ (\setcomp,\kappa)\text{ induces }\tau\}.$$ The idea is essentially a sign-reversing involution argument: $\theta$ changes the parity of $p$ for all elements of $\Pi_{(\kappa,\tau)}$ except one. This exceptional element is a fixed point of $\theta$, with $p = d$.

In order to define $\theta: \Pi_{(\kappa,\tau)} \rightarrow \Pi_{(\kappa,\tau)}$, fix $\kappa_\tau$ any proper coloring of $\Sigma$ inducing $\tau$. Define $\kappa_\tau'$ the coloring induced by a pair $((V_1,\dots,V_d),\kappa_\tau)$ as in equation \eqref{eq: def-coloring}, where each set $V_i$ in the decomposition contains exactly one element. The coloring $\kappa_\tau'$ is a proper coloring inducing $\tau$ and all the colors it uses are different in absolute value. Thus we may order the vertices $v$ in $V$ with respect to $|\kappa'_\tau(v)|$: 
$$|\kappa'_\tau(v_1)|<|\kappa'_\tau(v_2)|<\cdots<|\kappa'_\tau(v_d)|.$$
Now we are going to apply the following algorithm.

We consider the vertices with respect to the order $v_j$, starting from $v_1$. So let $v_j$ be an element of $V_i$.
\begin{itemize}
    \item If $|V_i| > 1$ then $\theta(\setcomp) = (V_1,\dots,V_{i-1},V_i\backslash \{v_j\}, \{v_j\},\dots,V_p)$.
    \item If $|V_i| = 1$ and if the following conditions hold:
    \begin{itemize}[label = $\ast$]
        \item $i>1$
        \item $\kappa|_{V_{i-1}\cup\{v_j\}}$ is proper
        \item $k>j$ for all $v_k\in V_{i-1}$,
    \end{itemize}
    then $\theta (\setcomp) = (V_1,\dots,V_{i-1}\cup \{v_j\}, V_{i+1},\dots, V_p)$.
    \item If not then go to the next vertex $v_{j+1}$, and if $v_j$ is the last one ($j= d$), then we set $\theta(\setcomp) = \setcomp$.
\end{itemize}

This assignment is well defined, this means $\setcomp \in \Pi_{\pair}$ implies $\theta(\setcomp) \in \Pi_{\pair}$. 
It is straightforward that the algorithm is designed in such a way that $\theta$ is an involution, which changes the parity of $p$ for every element unless in one case. 
The order on the vertices implies that if we have an improper edge for $\kappa$ between $v_j\in V_i$ and $v_k\in V_{i-1}$, then we have $j>k$.
Thus, there is only one element in $\Pi_{\pair}$ for which $\theta(\setcomp) = \setcomp$ and it is the set composition $(\{v_1\},\dots,\{v_d\})$.

\end{proof}

\begin{remark}
Note that this proof also yields the analogous result for classical graphs, originally proved by Stanley in \cite{S95}. Specifically, the same argument shows that $\omega(X_G(x_1,x_2,\dots)) = \sum_{(\kappa,\tau)}x_{\kappa(v_1)}\cdots x_{\kappa(v_d)}$ where the sum is taken over all pairs $(\kappa,\tau)$ where $\tau$ is an acyclic orientation of $G$ and $\kappa$ is a coloring compatible with $\tau$.
Stanley's original proof is different as it is based on the reciprocity theorem for P-partitions.
\end{remark}

The reciprocity result for symmetric functions implies, by specialization, a reciprocity result for the chromatic polynomial.

\begin{corollary}[\cite{Z82}, Theorem 3.5]
\label{coro:recip}
Given a signed graph $\Sigma = (V,E,\sigma)$ \begin{equation*}
\chi^b_\Sigma (-2t) =(-1)^{|V|}
\# \left\{
(\kappa,\tau) \mid
\kappa: V \to \{\pm 1,\dots,\pm t\},\;
\tau \text{ acyclic and compatible with } \kappa
\right\}.
\end{equation*}
\end{corollary}

\begin{proof}
The proof is straightforward taking into account that the antipode on the Hopf algebra of polynomials sends $P(t)$ to $P(-t)$.
\end{proof}
\section{Deformations of the character}
\label{sec: 4}

It is now well understood that computing the convolution power of a character is an powerful way to obtain polynomial invariants of combinatorial objects. Moreover, this approach allows to obtain combinatorial reciprocity results for these polynomials via the antipode. This question is deeply studied in \cite{AA23} in the context of Hopf monoids.
The aim of this section is to explore how the \textit{deformation} of a character can lead us to polynomial invariants in two variables.

Given a combinatorial Hopf algebra $(\mathcal{H}, \zeta)$, a \textit{deformation} of $\zeta$ is a morphism of $\mathbb{K}$-algebras $\zeta_x: \mathcal{H}\rightarrow \mathbb{K}[x]$ verifying $\zeta_x(h)|_{x = 0} = \zeta(h)$ for all $h\in \mathcal{H}$. Recall that for any character $\zeta$, $$\zeta^{\ast t}(h) := m^{(t-1)} \circ \zeta^{\otimes t}\circ \Delta^{(t-1)}(h) $$ is a polynomial in $t$ for any $h\in \mathcal{H}$. In the same way, one can consider the convolution power $\zeta_x^{\ast t}$ of the deformation of a character as $\zeta_x^{\ast t}(h) : = m^{(t-1)} \circ \zeta_x^{\otimes t}\circ \Delta^{(t-1)}(h)$ where in this case $m$ represents the product in $\K[x]$. Clearly $\zeta_x^{\ast t}(h) \in \K[x,t]$ for every $h\in \mathcal{H}$. This is how, by defining a deformation of a character, one can obtain a bivariate polynomial invariant.

The structure of this section is as follows. In Subsections \ref{subsec: whitney} and \ref{subsec: dichromatic}, we treat the case of (unsigned) graphs. In each of these subsections, we study a different deformation of $\zeta$, the classical character associated with graphs (see~\eqref{eq: zeta-graphs}). These deformations yield two distinct well-known bivariate polynomial invariants. Finally in \ref{subsec: deformation-signed-graphs} we treat the case of signed graphs.

\subsection{The Whitney number polynomial of (unsigned) graphs}
\label{subsec: whitney}
The first deformation that we consider was defined in \cite{BHM16}. The authors use it to define a deformation of Stanley's chromatic symmetric function as:
 \begin{equation}
 \label{eq: deformation-stanley}
 \Psi_x({G}) = \sum_{V = V_1 \sqcup \cdots \sqcup V_p}x^{\sum_i rk(G|_{V_i})}M_{(|V_1|,\dots, |V_p|)}.
 \end{equation}
 
  In this article, the authors provide an infinite family of graphs that can be distinguished by $\Psi_x$ but not by $\Psi$. We remark that we call \textit{deformations} what the authors call \textit{q-analogues}. We have chosen this terminology since the original character is recovered when the parameter is set to $0$. This deformation is defined as $\zeta_x (G) : = x^{rk(G)}$ for any graph $G$, where $rk(G)$ is the rank of the graphical matroid of $G$. That is, $rk(G)$ is the number of edges in any spanning forest of $G$. It is not dificult to check that $\zeta_x (G)$ is indeed a deformation of $\zeta$, the usual character of graphs.

The following definition is due to Zaslavsky. It first appeared in \cite{Z97} in a more general context of hyperplane arrangements under the name of \textit{Möbius polynomial}. 
\begin{definition} Given a graph $G = (V,E)$ define the Whitney number polynomial $w_G(x,t)$ as \begin{equation}
\label{eq: def-whitney}
    w_G(x,t) := \sum_{A\in \mathcal{F}(G)} x^{rk(A)}\chi_{G/A}(t)
\end{equation} where $\mathcal{F}(G)$ is the set of flats of $G$. In a graph $G = (V,E)$, a flat $A \in \mathcal{F}(G)$ is a subset of $E$ such that each connected component of the graph $(V,A)$ is an induced subgraph of $G$.
\end{definition}

\begin{remark}
Note that $w_G(0,t) = \chi_G(t)$. Also, the Whitney number polynomial is {\em not} the Whitney polynomial consisting in the evaluation in $(y-1,s-1)$ of the Tutte polynomial. 
\end{remark}

For our purpose, it is practical to write this polynomial as a generating function over all possible colorings. The following property was stated by Zaslavsky, we include a proof for the sake of completeness.

\begin{lemma} Given a graph $G = (V,E)$

\begin{equation}
    w_G(x,t) = \sum_{\kappa: V \rightarrow [t]}x^{rk(I(\kappa))}.
\end{equation}
\end{lemma}

\begin{proof}
    Note that for any coloring $\kappa$, $I(\kappa) \in \mathcal{F}(G)$. Indeed, suppose there is an edge $\{u,v\}\in E$ whose endpoints $u$ and $v$ are in the same connected component of $(V,I(\kappa))$. There must be a path of monochromatic edges between $u$ and $v$, so $\kappa(u) = \kappa(v)$ and then $\{u,v\}\in I(\kappa)$. Thus, 
    $$\sum_{\kappa \rightarrow [t]}x^{rk(I(\kappa))} = \sum_{A\in \mathcal{F}(G)}x^{rk(A)}\#\{\kappa: V \rightarrow [t]:~I(\kappa) = A\} = \sum_{A \in \mathcal{F}(G)}x^{rk(A)}\chi_{G/A}(t).$$
\end{proof}

In light of this lemma, it is easy to show that the image under $\phi_t$ of $\Psi_x(G)$ is exactly the Whitney number polynomial.

\begin{proposition}
    Given a graph $G$, consider the deformation $\zeta_x(G) = x^{rk(G)}$ then 
    \begin{equation}
        \label{eq:whitney-unsigned}
        w_G(x,t) = \phi_t \circ \Psi_x(G) = \zeta_x ^{\ast t}(G).
    \end{equation}
\end{proposition}

\begin{proof}
\begin{align*}
    \zeta_x^{\ast t}(G) = \sum_{V = V_1 \sqcup\cdots \sqcup V_t} \zeta_x(G|_{V_1})\cdots \zeta_x(G|_{V_t}) = \sum_{V = V_1 \sqcup\cdots \sqcup V_t}x^{rk(G|_{V_1}\sqcup \cdots \sqcup G|_{V_t})} = \sum_{\kappa: V \rightarrow [t]}x^{rk(I(\kappa))}
\end{align*}
\end{proof}

Given a graph $G$, by definition, $w_G(x,t)$ counts the number of colorings of certain contractions of $G$. But it also has interesting specializations. For instance the Whitney number polynomial specializes in the $f$-polynomial of the graphical zonotope $\mathcal{Z}[G]$. We explain this below.

Given a graph $G = ([d],E)$, denote $e_{i}$ the coordinate vector with all coordinates equal to 0 but having 1 in the position $i$. The \textit{graphical zonotope} of $G$ is the Minkowski sum $$\mathcal{Z}[G] = \sum_{\{i,j\}\in E} [e_i,e_j]$$ where $[e_{i},e_{j}]$ denotes the segment between $e_{i}$ and $e_{j}$.

Given a graph $G = (V,E)$ on $d$ vertices, Zaslavsky \cite{Z82} proved that:
\begin{equation}
    \sum_{i = 0}^d
f_i(\mathcal{Z}[G])t^i = (-1)^d w_G(-x,-1)
\end{equation} where $f_i(\mathcal{Z}[G])$ denotes the number of faces of dimension $i$ in $\mathcal{Z}[G]$.

Thus \eqref{eq:whitney-unsigned} may be used to compute 
the $f$-polynomial of graphical zonotopes.
We illustrate this on an example.

\begin{example} Let $G = \threepath$
 \begin{align*}
 \label{eq: zeta-alpha}
 w_G(x,t) = \zeta^{\ast t}_{x}(G) & =  \sum_{\alpha\models 3} \binom{t}{\ell(\alpha)}\sum_{\substack{V = V_1 \sqcup \cdots \sqcup V_{\ell(\alpha)} \\|V_i| = \alpha_i~ \forall i}} \zeta_{x}(G\mid_{V_1})\cdots \zeta_{x}(G\mid_{V_{\ell(\alpha)}})\\
  &  = \binom{t}{1} \zeta_x(G) + 2\binom{t}{2}(2 \zeta_x(\vertex ~ \positiveedge) + \zeta_x(\vertex ~\vertex ~ \vertex)) + 3! \binom{t}{3} \zeta_x(\vertex~\vertex ~\vertex)  \\
   & = x^2t + 2xt^2 + t^3 - 2xt - 2t^2 + t.
 \end{align*}

Evaluating in $x = 0$, one finds $\chi_G(t) = t^3-2t^2+t$, the chromatic polynomial of $G$. Also, $(-1)^3 w_G(-x,-1) = x^2 
+ 4x + 4$ is the $f$-polynomial of $\mathcal{Z}[G]$.
 
\end{example}

\subsection{The dichromatic polynomial of (unsigned) graphs}
\label{subsec: dichromatic}

In this subsection, we show how to obtain the dichromatic $D_G(u,t)$ polynomial as the power convolution of a suitable deformation of $\zeta$. The dichromatic polynomial is related to the Tutte polynomial $T_G(y,s)$ by the following change of variables:
\begin{equation}
    \label{eq: change-var-tutte}
T(y,s) = (s-1)^{-|V|}(y-1)^{-\cc(G)}D((s-1),(y-1)(s-1)).
\end{equation}

We follow Kauffman's presentation of the dichromatic polynomial \cite{kauffman1989tutte}. It is a well known fact that the chromatic polynomial of a graph $G$ is the polynomial $\chi_G(t)$ generated by the relations
\begin{align*}
& \chi_G(t) = \chi_{G \backslash e}(t)- \chi_{G /e}(t) \\
&\chi_{G\sqcup H}(t) = \chi_G(t)\cdot \chi_H(t),
\end{align*}
together with the condition that
 $\chi_{G}(t) = t$ if $G$ is the graph with a single vertex and no loops. Then one can obtain a generalization of the chromatic polynomial by adding an extra variable $u$ to these relations. The \textit{dichromatic polynomial} of $G$ is the polynomial $D_G(u,t)$ generated by the  relations:
\begin{equation}
\label{eq: relations}
\begin{aligned}
&D_G(u,t) = D_{G \backslash e}(u,t) + u D_{G /e}(u,t) \\
&D_{G\sqcup H}(u,t) = D_G(u,t)\cdot D_H(u,t),
\end{aligned}
\end{equation}
together with the condition that
$D_G(u,t) = t$ if $G$ is the graph with a single vertex and no loops. Clearly $D_G(-1,t) = \chi_G(t)$.

As in the preceding subsection, we are interested in writing this polynomial as a generating function over all possible colorings.

\begin{lemma}
\label{lemma: gen-dichrom}
    \begin{equation}
    \label{eq:gen-dichrom}
    D_G(u,t) = \sum_{\kappa: V \rightarrow [t]}(1+u)^{|I(\kappa)|}
    \end{equation}
\end{lemma}
\begin{proof}
    To see this, it suffices to show that the generating function satisfies the relations \ref{eq: relations} defining $D_G(u,t)$.
\end{proof}

\begin{proposition} In the CHA of graphs ($\mathcal{G},\zeta$), let $\zeta_x$ be the deformation of $\zeta$ defined by $\zeta_x(G) = x^{|E|}$ for $G = (V,E)$ a graph. Then
    \begin{equation}
    D_G(x-1,t) = \zeta^{\ast t}_x (G).
    \end{equation}
\end{proposition}

\begin{proof}
Clearly, $\zeta_x$ is a deformation of $\zeta$ and
\begin{align*}
    \zeta_x^{\ast t}(G) = \sum_{V = V_1 \sqcup\cdots \sqcup V_t} \zeta_x(G|_{V_1})\cdots \zeta_x(G|_{V_t}) = \sum_{V = V_1 \sqcup\cdots \sqcup V_t}x^{|G|_{V_1}\sqcup \cdots \sqcup G|_{V_t}|} = \sum_{\kappa: V \rightarrow [t]}x^{|I(\kappa)|}.
\end{align*}
\end{proof}

\begin{example}
    Consider the graph $G = \exampletutte$.
    \begin{align*}
    \zeta_x^{\ast t}(G) & =  \sum_{\alpha\models 3} \binom{t}{\ell(\alpha)}\sum_{\substack{V = V_1 \sqcup \cdots \sqcup V_{\ell(\alpha)} \\|V_i| = \alpha_i~ \forall i}} \zeta_{x}(G\mid_{V_1})\cdots \zeta_{x}(G\mid_{V_{\ell(\alpha)}}) \\
                     & = \binom{t}{1}\zeta_{x}(G) + 2 \binom{t}{2}(\zeta_{x} (\exampletutteuno) + \zeta_{x} (\exampletuttedos) + \zeta_{x} (\exampletuttetres)) + 3! \binom{t}{3}\zeta_{x}(\exampletuttedos)  \\
                     & = t x^4 + t(t-1)(x^2+x+x^3)  + t(t-1)(t-2)x.\end{align*}
Thus $D_G(u,t) = t (u+1)^4 + t(t-1)((u+1)+(u+1)^2+(u+1)^3)+t(t-1)(t-2)(u+1)$. By applying the change of variables in \ref{eq: change-var-tutte} the Tutte polynomial of $G$ is $$T_G(y,s) = y^2 s + s^2 y.$$
\end{example}

\begin{remark}
    In \cite{FM25} the authors use the formalism of double bialgebras to get an algebraic interpretation of a variant of the Tutte polynomial, namely the Fortuin-Kasteleyn polynomial. They deduce new simpler proofs of combinatorial results on these invariants.
\end{remark}

\subsection{Deformations in the CHA of signed graphs}
\label{subsec: deformation-signed-graphs}

We now explore how to use the same techniques to find bivariate polynomial invariants for signed graphs. The Whitney number polynomial of a signed graph was introduced by Zaslavsky in \cite{Z82}. It is defined in a complete analogous way as the classical case. Given a signed graph $\Sigma = (V,E,\sigma)$, the \textit{balanced} Whitney number polynomial is: 
$$w^b_\Sigma(x,2t) = \sum_{\kappa:V \rightarrow \{\pm 1,\dots, \pm t\}}x^{rk(I(\kappa))},$$ 
where for an edge set $A\subseteq E$, $rk (A) := |V|-\pi_b(\Sigma\vert A)$ with $\pi_b(\Sigma\vert A)$ the number of balanced components of $\Sigma\vert A : = (V,A,\sigma \vert_A)$. This polynomial satisfies an equation analogous to \eqref{eq: def-whitney} in the signed case. When evaluated at negative integers, it counts the number of acyclic orientations of certain
contractions of $\Sigma$ \cite[Corollary 3.6]{Z82}.
We are interested in expressing the balanced Whitney number polynomial as a convolution power of a certain deformation of the character $\zeta$.

\begin{definition}
Given a connected signed graph $\Sigma = (V,E, \sigma)$ and $S\subseteq V$, define $I_\Sigma(S)$ to be the set of edges $e\in E$ satisfying one of the following conditions:
\begin{itemize}
\item $\sigma(e)=+$ and the endpoints of $e$ are in $S$
\item $\sigma(e) = +$ and the endpoints of $e$ are in $V\backslash S$ 
\item $\sigma(e) = -$ and $e$ has one endpoint in $S$ and the other one in $V\backslash S$.
\end{itemize}
Then define $$\zeta_x(\Sigma) := \sum_{A \subseteq V} x^{rk(I_\Sigma(A))}. $$
\end{definition}

It is not difficult to check that $\zeta_x$ is a deformation of $\zeta$, when $x = 0$, the only terms that may survive are those having $I_\Sigma(A) = \emptyset$. This means that $\Sigma$ admits a decomposition as the one in Lemma \ref{lemma: decomposition} and hence it admits exactly two. Also, as given two disjoint set of edges $A_1$, $A_2 \subseteq E$, $rk(A_1\sqcup A_2) = rk(A_1) + rk(A_2)$, $\zeta_x$ is multiplicative.

\begin{proposition}
\label{prop: conv-zeta-signed}
The balanced Whitney number polynomial may be obtained by the convolution of the deformed character $\zeta_x$:
\begin{equation}
 w_\Sigma^b(x,2t) = \zeta^{\ast t}_x (\Sigma).
\end{equation}
\end{proposition}

\begin{proof}
We may write
$$\zeta^{\ast t}_x (\Sigma) = \sum_{V_1 \sqcup\cdots\sqcup V_t} \zeta_x (\Sigma |_{V_1}\sqcup \cdots \sqcup \Sigma |_{V_t}) = \sum_{V_1 \sqcup\cdots\sqcup V_t} \sum_{S \subseteq V} x^{rk(I_{\Sigma |_{V_1}\sqcup \cdots \sqcup \Sigma |_{V_t}}(S))}.$$

But a decomposition $V = V_1 \sqcup \cdots \sqcup V_t $ and a choice $S\subseteq V$ induces a coloring with color set $\{\pm 1,\dots, \pm t\}$ by letting $\kappa|_{V_i \cap S} = i$ and $\kappa|_{V_i \cap (E \backslash S)} = -i$ for all $i = 1,\dots t$. And this correspondence is bijective. It is not difficult to check that the improper set of this coloring is $I_{\Sigma |_{V_1}\sqcup \cdots \sqcup \Sigma |_{V_t}}(S)$.
\end{proof}

\begin{remark}
    The \textit{unbalanced} Whitney number polynomial is defined 
    in the same way as the balanced one:
    $$
    w_\Sigma(x,2t+1) = \sum_{\kappa:V \rightarrow \{0,\pm 1,\dots, \pm t\}}x^{rk(I(\kappa))}.
    $$
    The unbalanced Whitney number polynomial specializes in the $f$-polynomial of the associated zonotope \cite[Corollary 4.1'']{Z82}. But, as stated in the introduction, this case will be treated in a forthcoming work.
\end{remark}

\begin{example}

Note that for any signed graph $\Sigma$ 
$$\zeta_x(\Sigma) = \sum_{A \subseteq V} x^{rk(I_\Sigma(A))} = 2\cdot \sum_{\substack{A \subseteq V\\\vert A \vert \leq \vert V \vert /2}} x^{rk(I_\Sigma(A))}$$ as $I_\Sigma (A) = I_\Sigma(E \backslash A)$. Consider the signed graph $\Sigma = \othersignedtriangle$,

\begin{align*}
\zeta_x(\Sigma) = 2\cdot  ( x^{rk(\twopositiveedges)}+  x^{rk(\negativeedge)} +  x^{rk(\signedtriangle)} + x^{rk(\twosignededges)}) = 6x^2 + 2x.
\end{align*}
And the balanced Whitney number polynomial can be computed as

\begin{align*}
    w^b_\Sigma (x,2t) & =  \sum_{\alpha\models 3} \binom{t}{\ell(\alpha)}\sum_{\substack{V = V_1 \sqcup \cdots \sqcup V_{\ell(\alpha)} \\|V_i| = \alpha_i~ \forall i}} \zeta_{x}(\Sigma\mid_{V_1})\cdots \zeta_{x}(\Sigma\mid_{V_{\ell(\alpha)}}) \\
                     & = \binom{t}{1}\zeta_{x}(\Sigma) + 2 \binom{t}{2}(\zeta_{x} (\vertex ~ \positiveedge) + \zeta_{x} (\vertex ~ \negativeedge) + \zeta_{x} (\vertex ~ \positiveandnegativeedge)) + 3! \binom{t}{3}\zeta_{x}( \vertex ~ \vertex ~ \vertex)  \\
                     & = \binom{t}{1} (6x^2+2x) + 2 \binom{t}{2} (16x + 8)  + \binom{t}{3} 48.\end{align*}

\end{example}

In \cite{Z95} Zaslavsky defines the dichromatic polynomial for signed graphs, it satisfies an equation analogous to \eqref{eq:gen-dichrom} \cite[Corollary 4.4]{Z95}. This polynomial in two variables can be obtained as a convolution product of the following deformation
$$\zeta_x(\Sigma) := \sum_{A \subseteq V} x^{|I_\Sigma(A)|}.$$

We omit the details, since everything occurs exactly in the same way as in Proposition \ref{prop: conv-zeta-signed}.

We conclude with an open question.
In \cite{GLRV21}, the authors point out that the dichromatic polynomial defined by Zaslavsky, while it does specialize to the chromatic polynomial, does not specialize to the polynomial that counts the number of “nowhere-zero flows.” And because of this, it does not generalize the classical dichromatic polynomial. As a solution, they propose a Tutte polynomial in three variables that satisfies both specializations. Furthermore, this polynomial simultaneously handles the cases of colorings with {\em and} without a color $0$. It is certainly interesting to explore this polynomial from the perspective of signed combinatorial Hopf algebra with the techniques developed in this article.

\bibliographystyle{amsalpha}
\bibliography{refs}

\end{document}